\documentclass[12pt]{amsart}

\usepackage{amssymb,latexsym}
\usepackage{amsfonts}
\usepackage{amsmath}

\usepackage{systeme}

\usepackage[colorlinks,linkcolor=blue,anchorcolor=blue,citecolor=blue]{hyperref}

\newcommand\C{{\mathbb C}}

\newcommand\Q{{\mathbb Q}}
\newcommand\R{{\mathbb R}}
\newcommand\N{{\mathbb N}}
\newcommand\Z{{\mathbb Z}}

\newcommand\al{\alpha}
\newcommand\be{\beta}
\newcommand\ga{\gamma}
\newcommand\eps{\varepsilon}

\newtheorem{theorem}{Theorem}
\newtheorem{lemma}[theorem]{Lemma}

\newtheorem{corollary}[theorem]{Corollary}

\theoremstyle{remark}

\begin{document}

\title[Linear relations with conjugates of a Salem number]{Linear relations with conjugates of a Salem number}


\author{Art\= uras Dubickas}
\address{Department of Mathematics and Informatics, Vilnius University, Naugarduko 24,
LT-03225 Vilnius, Lithuania}
\email{arturas.dubickas@mif.vu.lt}

\author{Jonas Jankauskas}
\address{Mathematik und Statistik,
Montanuniversit\"at Leoben,
Franz Josef Strasse 18,
A-8700 Leoben,
Austria}
\email{jonas.jankauskas@gmail.com}

\subjclass[2010]{11R06, 11R09}

\keywords{Additive linear relations, Salem numbers, Pisot numbers, totally real algebraic numbers}

\begin{abstract}
In this paper we consider linear relations with conjugates of a Salem number $\al$. 
We show that every such a relation arises from a linear relation between conjugates of the
correspon\-ding totally real algebraic integer $\al+1/\al$. It is also shown that the smallest degree of a Salem number with a nontrivial relation between its conjugates is $8$, whereas the smallest length of a nontrivial linear relation between the conjugates of a Salem number is $6$.  
\end{abstract}

\maketitle

\section{Introduction}

Let $\al$ be an algebraic number of degree $d \geq 2$ over $\Q$ with conjugates $\al_1=\al,\al_2,\dots,\al_d$. An additive linear relation
\begin{equation}\label{rysys}
k_1\al_1+k_2\al_2+\dots+k_d\al_d=0
\end{equation}
with some $k_1,k_2,\dots,k_d \in \Q$ is called {\it nontrivial} if 
$k_i \ne k_j$ for some $1 \leq i<j \leq d$.
Thus, the relation
$${\rm Trace}(\al):=\al_1+\al_2+\dots+\al_{d} =0$$
and its rational multiples $\sum_{j=1}^d r \al_j=0$, where $r \in \Q$, are {\it trivial} linear relations. (These hold for conjugates of any algebraic number whose trace is zero.)
Note that if \eqref{rysys} is a nontrivial linear relation with some $k_1,k_2,\dots,k_d \in \Q$ then, by multiplying all the $k_i$ by their common denominator, we can assume that $k_1,k_2,\dots,k_d \in \Z$. Accordingly, we call the sum $$|k_1|+|k_2|+\dots+|k_d| \in \N$$ the {\it length}
of the relation \eqref{rysys}. 

The investigation of nontrivial linear relations \eqref{rysys} in conjugates of algebraic numbers 
has begun with the papers of Kurbatov \cite{kurb1, kurb2, kurb3}. In \cite{smy2}, Smyth obtained some useful results and also formulated several natural conjectures on the possibility of \eqref{rysys} which are still wide open; see also his previous paper \cite{smy1}. Further results on this subject have been obtained by several authors in \cite{bds,dix,ds1,ds2,dub,dub2,girst,fl1,fl2,vall}. 

Recently, in \cite{DHJ}
it was shown that there is a unique Pisot number $\al=(1+\sqrt{3+2\sqrt{5}})/2$ with minimal polynomial $x^4-2x^3+x-1$ satisfying the nontrivial linear relation $$\al_1+\al_2-\al_3-\al_4=0$$ of length $4$. 
Recall that
an algebraic integer $\al>1$ is called a {\it Pisot number} if its other conjugates over $\Q$ (if any) all lie in the open unit disc $|z| < 1$. 
This answers two questions raised earlier in \cite{smy0}. For instance, this implies that no two conjugates of a Pisot number can have the same imaginary part. See also a subsequent paper \cite{DJ} for some further analysis of some simple linear relations of small length.

In the present paper, we investigate additive linear relations in conjugates of a Salem number. Recall that
an algebraic integer $\al>1$ is called a {\it Salem number} if  its other conjugates over $\Q$ all lie in the closed unit disc $|z| \le 1$ with at least one conjugate lying on the circle $|z|=1$. 

Throughout, if $\al>1$ is a Salem number 
of degree $d=2s \geq 4$ we label its conjugates
as in the theorem below.

\begin{theorem}\label{pirmoji}
Let $\al_1=\al>1$ be a Salem number of degree $d=2s \geq 4$ with conjugates $\al_1,\dots,\al_d$ satisfying
$\al_2=1/\al_1$ and 
$\al_{2j}=1/\al_{2j-1}=\overline{\al_{2j-1}}$ for $j=2,\dots,s$. 
If for some rational numbers $k_i$, $i=1,\dots,d$, and for some totally real algebraic number $\gamma$ we have 
\begin{equation}\label{duu}
k_1\al_1+k_2\al_2+\dots+k_d\al_d=\gamma,
\end{equation}
then 
$k_{2j-1}=k_{2j}$ for each $j=1,\dots,s$.
\end{theorem}

In particular, the theorem obviously holds for $\gamma=0$. So, every linear relation \eqref{rysys} in the conjugates $\al_i$, $i=1,\dots,d$, of a Salem number $\al$ is induced by the linear relation 
\begin{equation}\label{rysys1}
k_1\be_1+k_3\be_2+\dots+k_{2s-1}\be_s=0
\end{equation}
in conjugates of the respective totally real algebraic integer $\be_1=\be:=\al+1/\al>2$ whose other conjugates
are
$$\be_j=\al_{2j-1}+\al_{2j}=\al_{2j-1}+1/\al_{2j-1}=\al_{2j-1}+\overline{\al_{2j-1}}=2\Re \al_{2j-1} \in (-2,2)$$ for $j=2,\dots,s$.  If $f$ is the minimal polynomial of a Salem number $\al$ of degree $d=2s$ and $g$ is the minimal polynomial of 
$\be=\al+1/\al$ of degree $s$ then they are related by the identity
$
f(x)=x^{s} g(x+1/x).
$
Then, as in \cite{chr}, we call $g$ the {\it trace polynomial} of $f$. 
Note that $f$ is irreducible if and only if $g$ is irreducible. Also,
${\rm Trace}(\al)=\sum_{j=1}^d \al_j= \sum_{i=1}^s \be_i= {\rm Trace}(\be)$. 

By \cite{kurb3} (or \cite{ds1}), the only relation with conjugates $\be_1=\be,\dots,\be_p$ 
of an irreducible polynomial of prime degree $p$ can be of the form 
$$r\be_1+r\be_2+\dots+r\be_p=0,$$ where $r \in \Q$. Hence, 
the only possible linear relation with conjugates of a Salem number $\al$
with degree $2p$
is $r {\rm Trace}(\al)=0$, where $r \in \Q$. This relation is trivial.

So, in particular Theorem~\ref{pirmoji} implies that 

\begin{corollary}\label{keturi}
If $p$ is a prime number then there are no nontrivial linear relations in conjugates of a Salem number of degree $d=2p$.  
\end{corollary}

By \cite{mcsm}, it is known that there are Salem numbers of every integral trace. The degree of a Salem number with negative trace 
$-t$ is quite large if $t \in \N$ is large. Earlier, in \cite{mcsm0} 
Smyth has shown that there are Salem numbers with trace $-1$ of every even degree
$d \geq 8$. 

Here, by a similar argument, we show that

\begin{theorem}\label{keturi-1}
For any even $d \geq 6$ there is a Salem number of degree $d$
with trace $0$.    
\end{theorem}

In Corollary~\ref{hbhbhb} below, we list of all $4$ possible Salem numbers of degree $6$ and trace $0$. Note that there are no Salem numbers 
of degree $4$ and trace $0$. Indeed, otherwise the minimal polynomial of
such a Salem number would be $x^4+ax^2+1$, with $a \in \Z$, which is impossible. 

Our next theorem describes the minimal length  of nontrivial linear relations between conjugates of a Salem number and the minimal degree
of a Salem number for which a nontrivial linear relation may occur. 

\begin{theorem}\label{antroji}
Suppose $\al>1$ is a Salem number with conjugates $\al_1=\al,\al_2,\dots, \al_d$ over $\Q$ labelled as in Theorem~\ref{pirmoji}. 
\begin{itemize}
\item[$(i)$]
If for some
integers  $k_1,k_2,\dots,k_d$, not all zero, the nontrivial linear relation \eqref{rysys} holds then its length 
must be at least $6$.
Furthermore, there exist Salem numbers $\al$ of degree $12$ whose six conjugates satisfy the following nontrivial linear relation of length $6$:
$$
\al_1+\al_2+\al_3+\al_4+\al_5+\al_6=0.
$$

\item[$(ii)$] The smallest degree of a Salem number with a nontrivial linear relation between its conjugates is $8$.  Furthermore, there exist Salem numbers $\al$ of degree $8$ whose conjugates satisfy the following nontrivial linear relation:
$$
\al_1+\al_2+\al_3+\al_4-\al_5-\al_6-\al_7-\al_8=0.
$$
\end{itemize}
\end{theorem}

\section{Auxiliary results}

We begin with two simple lemmas. 

\begin{lemma}\label{trecioji}
The cubic polynomial $x^3-ax+b \in \R[x]$ has three distinct roots in the interval $(-2,2)$ iff $0<a<4$ and
\begin{equation}\label{pima}
 \max\Big(2a-8, -\frac{2a\sqrt{a}}{3\sqrt{3}}\Big)<b < \min\Big(8-2a, \frac{2a\sqrt{a}}{3\sqrt{3}}\Big),
\end{equation}
and two distinct roots in $(-2,2)$ and one root in $(2,+\infty)$ iff $3<a<12$
and
\begin{equation}\label{pima2}
-\frac{2a\sqrt{a}}{3\sqrt{3}}< b < -|2a-8|.
\end{equation}
\end{lemma}

\begin{proof}
Set $$h(x):=x^3-ax+b.$$Since
$h'(x)=3x^2-a$, the polynomial $h$ has only one real root 
if $a \leq 0$. 

Suppose $a>0$.  Set $x_0:=\sqrt{a/3}$. 
Then, the polynomial $h$
has three distinct roots in $(-2,+\infty)$ iff $-2<-x_0$ (i.e., $0<a<12$), 
\begin{equation}\label{mm1}
h(-2)=-8+2a+b<0, 
\end{equation}
\begin{equation}\label{mm2}
h(-x_0)=\frac{2a\sqrt{a}}{3\sqrt{3}}+b>0
\end{equation}
and
\begin{equation}\label{mm3}
h(x_0)=-\frac{2a\sqrt{a}}{3\sqrt{3}}+b<0.
\end{equation}

Clearly, all three roots belong to $(-2,2)$ if, in addition, we have $h(2)=8-2a+b>0$. Combined with \eqref{mm1}, \eqref{mm2} and \eqref{mm3}
this proves \eqref{pima}. Evidently, \eqref{pima} is only possible for some $b$ when its left hand side does not exceed its right hand side, that is, when $0<a<4$.

Similarly, two roots of $h$ are in $(-2,2)$ and one root in $(2,+\infty)$ when one has \eqref{mm1}, \eqref{mm2}, and
$h(2)=8-2a+b<0$. (As $h$ is increasing in the interval $(x_0,2)$, the inequality \eqref{mm3} automatically holds.) Evidently, all these inequalities combine into \eqref{pima2}. Here, as $0<a<12$, it is easy to see that the inequality $$|2a-8|<\frac{2a\sqrt{a}}{3\sqrt{3}}$$ holds only for $a$ in the range $3<a<12$, so only for such $a$ one can find some $b$ satisfying
\eqref{pima2}.
\end{proof}

Observe that there are only $7$ pairs of integers $(a,b)$ satisfying the conditions $3<a<12$ and \eqref{pima2}, namely,
$(4,-1)$, $(4,-2)$, $(4,-3)$, $(5,-3)$, $(5,-4)$, $(6,-5)$ and $(7,-7)$. 
However, the polynomials $x^3-4x-3$, $x^3-5x-4$ and 
$x^3-6x-5$ are reducible. Other four polynomials $x^3-4x-1$, $x^3-4x-2$, $x^3-5x-3$ and 
$x^3-7x-7$ are irreducible. So, Lemma~\ref{trecioji} implies that 

\begin{corollary}\label{hbhbhb}
There are exactly four Salem numbers of degree $6$ with trace $0$. 
Their minimal polynomials are: $$x^6-x^4-x^3-x^2+1, \quad
x^6-x^4-2x^3-x^2+1,$$ 
$$x^6-2x^4-3x^3-2x^2+1, \quad x^6-4x^4-7x^3-4x^2+1.$$ 
\end{corollary}

\begin{lemma}\label{trecioji-1}
Let $h(x) \in \Z[x]$ be a monic polynomial of degree $k \geq 2$ with 
$k-1$ roots in the interval $(-2,1/4)$ and one root in $(-6,-2)$.  Then,
$$f(x):=(-1)^k x^{2k} h\big((x+1/x)(1-x-1/x)\big) \in \Z[x]$$ is a monic reciprocal polynomial of degree $4k$ which defines a Salem number of degree $d=4k$ in case it is irreducible over $\Q$. Moreover, the conjugates of this Salem number $\al$ labelled as in Theorem~\ref{pirmoji}
satisfy 
\begin{equation}\label{aaa}
\al_1+\al_2+\al_3+\al_4=\dots=\al_{4k-3}+\al_{4k-2}
+\al_{4k-1}+\al_{4k}=1.
\end{equation}
\end{lemma}

\begin{proof}
Let $\ga_1 \in (-6,-2)$ and $\ga_2<\dots<\ga_{k} \in (-2,1/4)$ be the roots of $h$. Consider the monic polynomial $g(x):=(-1)^k h(x(1-x))$. Then, its roots
\begin{equation}\label{bbbb}
\be_{2j-1}:=\frac{1 + \sqrt{1-4\ga_j}}{2} \quad 
\text{and} \quad
\be_{2j}:=\frac{1 - \sqrt{1-4\ga_j}}{2},
\end{equation}
where $j=1,\dots,k$,
satisfy $\be_1=(1+\sqrt{1-4\ga_1})/2>2$, $\be_2=(1-\sqrt{1-4\ga_1})/2 \in (-2,-1)$ and $\be_3,\dots,\be_{2k} \in (-1,2)$. So, $g$ has $2k-1$
roots in $(-2,2)$ and one root greater than $2$. Clearly, by \eqref{bbbb}, we have
\begin{equation}\label{bee}
\be_{1}+\be_2=\dots=\be_{2k-1}+\be_{2k}=1.
\end{equation}

Now, as the roots $\al_1=\al>1,\al_2=1/\al,\dots,\al_{4k-1},\al_{4k}=1/\al_{4k-1}$ of $$f(x)=x^{2k}g(x+1/x)=(-1)^k x^{2k}h\big((x+1/x)(1-x-1/x)\big)$$ satisfy  $\be_j=\al_{2j-1}+\al_{2j}=\al_{2j-1}+1/\al_{2j-1}$ for each $j=1,\dots,2k$, we see that \eqref{bee} implies \eqref{aaa}. Furthermore,   $\al$ is a Salem number of degree $4k$ provided that $f$ is irreducible
over $\Q$. 
\end{proof}

We made some calculations related to Lemma~\ref{trecioji-1}. It turns out that there exactly $15$ quadratic polynomials $h$ satisfying the conditions of the lemma and thus producing $15$ Salem numbers 
of degree $8$ satisfying \eqref{aaa} with $k=2$. For instance, $x^2+4x+1$ is such a quadratic polynomial $h$. Also, there are exactly 
$30$ cubic, $20$ quartic and $4$ quintic polynomials $h$ producing
$30$ Salem numbers of degree $12$ (satisfying \eqref{aaa} with $k=3$), $20$ Salem numbers of degree $16$ (satisfying \eqref{aaa} with $k=4$) and $4$ Salem numbers of degree $20$ (satisfying \eqref{aaa} with $k=5$), respectively.  In the case $k=5$, the example of $h$ is  $$x^5+9x^4+22x^3+16x^2-x-1.$$ This gives a Salem number $\al$ of degree $20$ with minimal polynomial 
$$x^{20}-5x^{19}+11x^{18}-19x^{17}+26x^{16}-29x^{15}+
27x^{14}-19x^{13}+8x^{12}+x^{11}$$ $$-5x^{10}+
x^9+8x^8-19x^7+27x^6-29x^5+
26x^4-19x^3+11x^2-5x+1$$
whose conjugates satisfy \eqref{aaa} with $k=5$. 

The first part of the next lemma was inspired by Lemma 1 of Beukers and Smyth in \cite{beusmy}. Essentially, it is  a version of their algorithm \cite{beusmy} to locate cyclotomic points on curves, specialized to the case of sequences of polynomials that produce Salem numbers from Pisot numbers. Also, the second part of Lemma \ref{Lemma_cyclotomic_factors} is loosely related to the work on irreducibility of polynomials of the type $x^nf(x)+g(x) \in \Z[x]$ and on the sequences  and covering systems of integers by Schinzel \cite{schi}, Filaseta et al. \cite{fifoko, fima}, although these irreducibility results are not of direct relevance here. Throughout, $f^*(x) = x^{\deg{f}}f(1/x)$ stands for the \emph{reciprocal polynomial} of $f(x)$.

\begin{lemma}\label{Lemma_cyclotomic_factors}
For $n \in \N$, consider the sequence of polynomials
\[
g_n(x) := x^nf(x) + \eps f^{*}(x),
\] where $\eps \in \{-1, 1\}$ and $f(x) \in \Z[x]$ satisfies $f^*(x) \ne \pm f(x)$.  
Suppose that a root of unity $\zeta \in \C$ is also a root of some polynomial $g_n(x)$. Then, $\zeta$ must appear among the zeros of at least one of the following polynomials:
\[
f(x^2)f^*(x)^2 + \eps f(x)^2f^*(x^2), \qquad f(x)^2f^*(-x^2) \pm  f(-x^2)f^*(x)^2,
\]
\[
f(x)f^*(-x) \pm  f(-x)f^*(x).
\]
In particular, if none of these polynomials is identically zero, then the set of all such possible roots of unity $\zeta$ is finite.

In addition to this, if $f(\zeta) \ne 0$ then the root of unity $\zeta$ is a zero of $g_n(z)$ if and only if $n$ belongs to the arithmetic progression $\ell k+r$, $k=0, 1, 2, \dots$, where $r$ is a fixed integer in the range $0 \leq r < \ell$ and $\ell={\rm ord}(\zeta)$ denotes the multiplicative order of $\zeta$. 
\end{lemma}

\begin{proof}
As $\zeta$ is the root of unity, by Lemma 1  of \cite{beusmy} (or Lemma 2.1 of \cite{mcsm0}, at least one of the three numbers $\zeta^2$, $-\zeta^2$, $-\zeta$ must be an algebraic conjugate of $\zeta$ over $\Q$.
Multiplying $g_n(x)=x^n f(x)+\eps f^*(x)$ by $x^n f(x)-\eps f^*(x)$ we see that the polynomial $h(x)=x^{2n}f(x)^2 - f^*(x)^2$ has a zero at $x=\zeta$.

If $\zeta^2$ is conjugate of $\zeta$, then one also has $g_n(\zeta^2)=0$. Combining this with $h(\zeta)=0$ yields
\[
\left\{
\begin{split}
\zeta^{2n}f(\zeta)^2 	&-  f^*(\zeta)^2		&= 0,\\
\zeta^{2n}f(\zeta^2)	&+ \eps f^*(\zeta^2)	&= 0.
\end{split}
\right.
\]
Hence,
\[
\begin{vmatrix}
f(\zeta)^2 	 &- f^*(\zeta)^2\\
f(\zeta^2) &\eps f^*(\zeta^2)\\
\end{vmatrix} = \eps f(\zeta)^2f^*(\zeta^2)  + f(\zeta^2)f^*(\zeta)^2 = 0.
\]
Thus, $\zeta$ is the root of $f(x^2)f^*(x)^2 + \eps f(x)^2f^*(x^2)$.

Suppose next that $-\zeta^2$ is a conjugate to $\zeta$. Then, using $g_n(-\zeta^2)=0$ and $h(\zeta)=0$, one concludes that $\zeta$ is the root of the polynomial
 $f(x)^2f^*(-x^2) +\eps (-1)^n f(-x^2)f^*(x)^2$. 

In the case when $-\zeta$ is conjugate to $\zeta$, from $g_n(\zeta)=g_n(-\zeta)=0$ one obtains $\zeta^n f(\zeta)+\eps f^*(\zeta)=0$ and $(-\zeta)^n f(-\zeta)+\eps f^*(-\zeta)=0$, which yields that $\zeta$ is a root of $f(x)f^*(-x) +(-1)^{n+1} f(-x)f^*(x)$.
 
Finally, if a root of unity $\zeta$ of order $\ell$ satisfies $g_n(\zeta)=0$, then $g_{n+\ell}(\zeta)=0$. Furthermore, if $\zeta$ is a common root of $x^{n_1}f(x) +\eps f^*(x)$ and $x^{n_2}f(x)+ \eps f^*(x)$, then $(\zeta^{n_2}-\zeta^{n_1})f(\zeta)=(\zeta^{n_2-n_1}-1)\zeta^{n_1}f(\zeta)=0$. By $f(\zeta) \ne 0$, it follows that $\zeta^{n_2 - n_1}=1$. Thus, $\ell \mid (n_2 - n_1)$ and so all such $n$ form an arithmetic progression with difference $\ell$, as claimed. 
\end{proof}

\section{Proofs of the theorems}

\begin{proof}[Proof of Theorem~\ref{pirmoji}]
Assume that $k_{2i} \ne k_{2i-1}$ for some $i$ in the range $1 \le i \le s$. Let $G$ be the Galois group of the normal extension of $\Q(\al,\ga)$ over $\Q$, and let $\sigma$ be an automorphism of $G$ which maps $\al_{2i-1}$ to $\al_1=\al$.  Then, $\sigma(\al_{2i})=\sigma(1/\al_{2i-1})=1/\al$, so that \eqref{duu} maps into
\begin{equation}\label{kjh}
\sigma(\gamma)=k_{2i-1}\al+k_{2i}/\al+ t_3\al_3+\dots+t_d\al_d,
\end{equation}
where $t_3,\dots,t_d \in \Q$ is a permutation of the list obtained from the initial list 
$k_1,\dots,k_d$ by excluding the elements  $k_{2i-1}$ and $k_{2i}$. 

Consider the following equality which is complex conjugate to \eqref{kjh}:
\begin{equation}\label{kjh1}
\overline{\sigma(\gamma)}=k_{2i-1}\al+k_{2i}/\al+ t_3\overline{\al_3}+\dots+t_d\overline{\al_d}.
\end{equation}
Since $\overline{\sigma(\ga)}=\sigma(\ga)$ and $\al_{2j}=\overline{\al_{2j-1}}$ for $j=2,\dots,s$, by adding \eqref{kjh} and \eqref{kjh1}, we obtain
$$
2\sigma(\gamma) = 2k_{2i-1}\al+2k_{2i}/\al+ w_2(\al_3+\al_4)+
\dots+w_{s}(\al_{d-1}+\al_d),
$$
where $w_{j}=t_{2j-1}+t_{2j}$ for $j=2,3,\dots,s$. 
Adding $2(k_{2i}-k_{2i-1})\al$ to both sides we deduce that
$$2\sigma(\gamma)+ 2(k_{2i}-k_{2i-1})\al=w_1(\al_1+\al_2)+w_2(\al_3+\al_4)+\dots+w_{s}(\al_{d-1}+\al_d),$$
where $w_1=2k_{2i}$.

As we already observed above, the number $\be_1=\be=\al+1/\al=\al_1+\al_2$ is totally real with conjugates $\be_2=\al_3+\al_4$, \dots, $\be_s=\al_{d-1}+\al_d$. Hence,
the number 
$$2(k_{2i}-k_{2i-1})\al=w_1\be_1+w_2\be_2+\dots+w_{s}\be_s-2\sigma(\gamma)$$
is a linear form (with rational coefficients $w_1,\dots,w_s,-2$) in totally real algebraic numbers $\be_1,\dots,\be_s,\sigma(\gamma)$.
Thus, it must be totally real itself. However, the number $2(k_{2i}-k_{2i-1})\al \ne 0$ is
not totally real, since it has a non-real conjugate $2(k_{2i}-k_{2i-1})\al_3$. 
This is a contradiction which completes the proof of the theorem. 
\end{proof}

\begin{proof}[Proof of Theorem \ref{keturi-1}] Assume that there exists a smallest even degree $d$ (where $d \geq 8$ by Corollary~\ref{keturi-1}), such that there are no Salem numbers of that degree $d$ with trace $0$. We will track down and ultimately eliminate all such possible $d$ by considering 3 sequences of polynomials, given explicitly by Salem's original construction \cite{salem1, salem2}.

We start with a Salem sequence 
\[
g_n(x)=x^n(x^3-x-1) + (-x^3-x^2+1), \quad n \geq 2.
\] Then $g_n(x)$ either posseses cyclotomic factors or it is a minimal polynomial of a Salem number of trace $0$; see \cite{boyd, salem1, salem2}. Since we have assumed that no Salem number of degree $d$ and trace $0$ exists, the polynomial $g_n(x)$ of degree $d=\deg{g_n}=n+3$ must be reducible, that is, it must be divisible by a cyclotomic polynomial $\Phi_{\ell}(x)$, where $\ell \in\N$.

To find cyclotomic factors of $g_n(x)$, we apply
Lemma~\ref{Lemma_cyclotomic_factors} with $f(x)=x^3-x-1$ and $\eps=1$. The following candidates appear as factors of auxiliary polynomials described in
Lemma~\ref{Lemma_cyclotomic_factors} (with $\eps=1$):
\[
\Phi_1(x) = x-1, \qquad \Phi_2(x)=x+1, \qquad \Phi_8(x)=x^4+1\]
\[
\Phi_{12}(x)=x^4-x^2+1, \qquad \Phi_{18}(x)=x^6-x^3+1,
\]
\[
\Phi_{30}(x) = x^8 + x^7 - x^5 - x^4 - x^3 + x + 1.
\]  Since none of the five auxiliary polynomials is zero identically, this list is complete. 

To see which of these candidates actually show up, one can apply the periodicity property stated in the second part of Lemma~\ref{Lemma_cyclotomic_factors}. After computation of ${\rm gcd}(g_n(x), \Phi_{\ell}(x))$, $0 \leq n \leq \ell-1$ for $\ell=1, 2, 8, 12, 18, 30$ it turns out that $g_n(x)$ has cyclotomic factors precisely for the degrees $d=n+3$ in one of the arithmetic progressions:
\[
d \in \{2k+1, 8k+2, 12k+1, 18k+17, 30k+24\},
\]
where $k=0, 1, 2, \dots $. As $d$ must be even, we restrict all such possible $d$ to two arithmetic progressions: $d \in \{8k+2, 30k+24\}$.

Next, we take the second sequence
\[
h_n(x) = \frac{x^n(x^2-x-1) - (-x^2-x+1)}{x-1},  \quad n \geq 2.
\]
Although now $f(x)=x^2-x-1$ contributes the coefficient $-1$ of $x^{n+1}$ to $g_n(x)$, one regains trace $0$ after division by $x-1$. 
Let us apply 
Lemma~\ref{Lemma_cyclotomic_factors} to the polynomial $g_n(x)=(x-1)h_n(x)$ with this new choice of $f(x)$ and $\eps=-1$. The candidate cyclotomic factors are:
\[
\Phi_1(x) = x - 1, \qquad \Phi_2(x) = x + 1,\qquad \Phi_3(x)= x^2 + x + 1,
\]
\[
\Phi_6(x)= x^2 - x + 1,\qquad \Phi_{12}(x)=x^4 - x^2 + 1.
\] As above, the computation of gcd's with first $12$ polynomials of the sequence yields the list of possible bad degrees $d=n+1$:
\[
d \in \{ 2k+1, 3k+ 2, 6k+3, 12k +4\}.
\]
This list also accounts for the single occurrence of the multiple factors, namely, $(x-1)^2$ in $g_4(x)$. Bad degrees must be even, so we are left with $d \in \{6k+2, 12k+4\}$. 

Let us combine this with the arithmetic progressions obtained from the first sequence:
\[
d \in \{8k+2, 30k+24\} \cap \{6k+2, 12k+4\}.
\]
Notice that all integers $30k+24$ are divisible by $6$, while none of $6k+2$ or $12k+4$ are. Therefore, $d \notin \{30k+24\}$, and hence $d \in \{8k+2\}$. Next, notice that $12k+4$ is divisible by $4$ while $8k+2$ is not. Consequently, $d \notin \{12k+4\}$. It follows that
\[
d \in \{8k+2\} \cap \{6k+2\} = \{24k + 2\}.
\]

To eliminate this possibility,  let us consider the third sequence, constructed with $f(x)=x^3-x^2-1$ and $\eps=-1$:
\[
h_n(x) = \frac{x^n(x^3-x^2-1) - (-x^3-x+1)}{x-1}, \quad n \geq 2.
\]
This time, by Lemma~\ref{Lemma_cyclotomic_factors} the candidates for cyclotomic divisors are
\[
\Phi_1(x)= x - 1, \quad
 \Phi_2(x)= x + 1, \quad
 \Phi_3(x)= x^2 + x + 1, \quad
 \Phi_4(x)= x^2 + 1, \quad
 \]
 \[
 \Phi_6(x)= x^2 - x + 1, \quad
 \Phi_{10}(x)= x^4 - x^3 + x^2 - x + 1, \quad
 \Phi_{18}(x)= x^6 - x^3 + 1.
\]
Now, bad degrees $d=n+2$ for this sequence $h_n(x)$ are
\[
d \in \{2k+1, 3k+1, 4k+3, 6k+4, 10k+5, 18k+ 6\}.
\] This last list accounts for the factor $(x-1)^2$ of $g_5(x)$ for a single value $n=5$. Since $d$ is even, $d \notin \{2k+1, 4k+3, 10k+5\}$. Since $d \in \{24k+2\}$ would have remainder $2 \pmod{3}$, we have $d \notin \{3k+1, 6k+4\}$. Finally, $d \notin \{18k+6\}$, since $24k+2$ is not divisible by $6$. This exhausts the list of possibilities, so no such bad degrees can exist. Hence, for each even $d \geq 6$, we can find a Salem number of degree $d$ and trace $0$ in one of the three Salem sequences that were considered above.
\end{proof}

\begin{proof}[Proof of Theorem~\ref{antroji}]
Suppose that the relation \eqref{rysys} holds with some $k_j \in \Z$, not all zero, and conjugates $\al_j$ of a Salem number $\al$ labelled as in Theorem~\ref{pirmoji}.  Then, by Theorem~\ref{pirmoji}, we 
have $k_{2j}=k_{2j-1}$ for $j=1,\dots,s$. Setting $\beta_j=\al_{2j-1}+1/\al_{2j-1}$ for $j=1,\dots,s$ we find that
\eqref{rysys1} holds, namely,
$k_{1}\beta_1+k_3\beta_2+\dots+k_{2s-1}\beta_s=0$.

In order to prove the first part of the theorem we need to show that $|k_1|+|k_3|+\dots+|k_{2s-1}| \geq 3$. 
For a contradiction, assume that $$|k_1|+|k_3|+\dots+|k_{2s-1}| \leq 2.$$
The case when $|k_{2j-1}|=2$ for some $j$ (and so other $k_{2i-1}$ are all zeros) is clearly impossible, since 
$\pm 2 \be_j \ne 0$. Therefore, we must have $|k_{2i-1}|=|k_{2l-1}|=1$, where $i<l$, and 
$k_{2j-1}=0$ for each $j \ne i,l$. Dividing both sides of the relation $k_{2i-1}\be_{i}+k_{2l-1}\be_l=0$ by $k_{2i-1}$, we find that $\be_{i}=-k_{2l-1}\be_{l}/k_{2i-1}=\pm \be_{l}$. Since $\be_{i} \ne \be_l$, the only possibility is $\beta_{i}=-\be_{l}$. Applying to it any automorphism $\sigma$ that maps $\be_{i}$ to $\be_1>2$ one obtains $\be_1=-\sigma(\be_l)$. Here, the left hand side is a real number greater than $2$, whereas the right hand side belongs to the interval $(-2,2)$, which is a contradiction. 

In order to prove the existence of a Salem number of degree $12$ with 
required linear relation among its conjugates we can
take, for instance, the following two pairs of real numbers $(a,b)$: 
$$(a_1,b_1)=(5-\sqrt{2}, -3+2\sqrt{2}) \quad \text{and} \quad  (a_2,b_2)=(5+\sqrt{2},-3-2\sqrt{2}).$$ 

Here, the first pair $(a_1,b_1)$ satisfies $0<a_1<4$ and \eqref{pima}, since $b_1=-0.171572\dots$ and the left and right hand sides of \eqref{pima}
are
$-0.828427\dots$ and $0.828427\dots$, respectively.
Thus, by Lemma~\ref{trecioji}, $x^3-a_1x+b_1$ has three roots in $(-2,2)$. 

The second pair $(a_2,b_2)$ satisfies $3<a_2<12$ and
\eqref{pima2}, because $b_2=-5.828427\dots$ and the left and right hand sides of \eqref{pima2}
are
$-6.252637\dots$ and $-4.8284427\dots$, respectively.
Hence, by Lemma~\ref{trecioji}, $x^3-a_2x+b_2$ has two roots in $(-2,2)$ and one greater than $2$.  

Consequently, their product
\begin{align*}
g(x) :&=(x^3-a_1x+b_1)(x^3-a_2x+b_2) \\&= (x^3-5x-3+\sqrt{2}(x+2))(x^3-5x-3-\sqrt{2}(x+2)) \\&=
x^6-10x^4-6x^3+23x^2+22x+1
\end{align*}
has $5$ roots in $(-2,2)$ and one greater than $2$. Now,
$$f(x):=x^6g(x+1/x)$$
equals to
$$x^{12}-4x^{10}-6x^9-2x^8+4x^7+7x^6+4x^5-2x^4-
6x^3-4x^2+1.$$
This polynomial defines a Salem number $\al=2.502568\dots$, since $f$ is irreducible over $\Q$. 

We remark than none of the choices with $\sqrt{2}$ replaced by
$\sqrt{3}$ or $\sqrt{5}$ works.  The pairs $(a_1,b_1)=(5-\sqrt{3},-3+2\sqrt{3})$ and $(a_2,b_2)=(5+\sqrt{3},-3-2\sqrt{3})$ satisfy the requirements of Lemma~\ref{trecioji}. However, the polynomial $g$ (and so $f$) is reducible:
\begin{align*}
g(x):&= (x^3-5x-3+\sqrt{3}(x+2))(x^3-5x-3-\sqrt{3}(x+2)) \\&
=x^6-10x^4-6x^3+22x^2+18x-3 \\&
=(x^2-3)(x^4-7x^2-6x+1).
\end{align*}
Similarly, with the pairs $(a_1,b_1)=(5-\sqrt{5},-3+2\sqrt{5})$ and $(a_2,b_2)=(5+\sqrt{5},-3-2\sqrt{5})$ one also obtains $g$ with $5$ roots in $(-2,2)$ and one in $(2,+\infty)$, but $g$ (and so $f$) is reducible:
\begin{align*}
g(x):&= (x^3-5x-3+\sqrt{5}(x+2))(x^3-5x-3-\sqrt{5}(x+2)) \\&
=x^6-10x^4-6x^3+20x^2+10x-11 \\&
=(x^2+x-1)(x^4-x^3-8x^2+x+11).
\end{align*}

By Corollary~\ref{keturi}, there no Salem numbers of degree $4$ or $6$
with a nontrivial linear relation among its conjugates. 
To give the example of a Salem number of degree $8$ with nontrivial
linear relation among its conjugates 
we can take, for instance,
$h(x):=x^2+4x+1$ with roots $\ga_1=-2-\sqrt{3}$ and $\ga_2=-2+\sqrt{3}$ satisfying the conditions of 
Lemma~\ref{trecioji-1}.  
Then, 
\begin{align*}
f(x):&=x^4 h\big((x+1/x)(1-x-1/x)\big) \\&=
x^8-2x^7+x^6-2x^5+x^4-2x^3+x^2-2x+1
\end{align*}
is irreducible. Hence, by Lemma~\ref{trecioji-1}, $f$ defines a Salem number 
$\al=1.994004\dots$ of degree $8$ whose conjugates satisfy \eqref{aaa}
with $k=2$.

As above, not every choice of an irreducible $h$ produces the irreducible polynomial $f$.
For example, selecting
$h(x):=x^2+4x+2$ whose roots $\ga_1=-2-\sqrt{2}$ and $\ga_2=-2+\sqrt{2}$ satisfy the conditions of 
Lemma~\ref{trecioji-1}, we get
the polynomial
\begin{align*}
f(x):&=x^4 h\big((x+1/x)(1-x-1/x)\big) \\&=
x^8-2x^7+x^6-2x^5+2x^4-2x^3+x^2-2x+1 \\&=
(x^4+1)(x^4-2x^3+x^2-2x+1)
\end{align*}
which is reducible.
\end{proof}

{\bf Acknowledgements.} 
The reserach of the first named author was funded by the European Social Fund according to the activity ‘Improvement of researchers’ quali\-fication by implementing 
world-class R\&D projects’ of Measure  No. 09.3.3-LMT-K-712-01-0037. 
The post-doctoral position of the second named author is supported by the Austrian Science Fund (FWF) project M2259 Digit Systems, Spectra and Rational Tiles under the Lise Meitner Program.


\begin{thebibliography}{99}

\bibitem{bds}
{\sc G.~Baron, M.~Drmota and M.~Ska{\l}ba,} {\it Polynomial relations of polynomial roots}, J. Algebra {\bf 177} (1995), 827--846.

\bibitem{beusmy}
{\sc F.~Beukers and C.~J.~Smyth,} {\it Cyclotomic points on curves}. In: Number Theory for the Millenium I, A.K. Peters 2002 ( Proceedings of the Millennial Conference on number theory, Urbana May 21--26, 2000, 67--85.

\bibitem{boyd}
{\sc D.~W.~Boyd}, {\it Small Salem numbers}, Duke Math. J. {\bf 44} (2) (1977), 315--328.

\bibitem{chr}
 {\sc C.~Christopoulos and J.~McKee,} {\it Galois theory of Salem polynomials,} Math. Proc. Camb. Phil. Soc. {\bf 148} (2010), 47--54.

\bibitem{dix}
{\sc J.~D.~Dixon,} {\it Polynomial relations of polynomial roots}, 
Acta Arith. {\bf 82} (1997), 293--302.


\bibitem{ds1}
{\sc M.~Drmota and M.~Ska{\l}ba,} {\it On multiplicative and linear independence of polynomial roots}, in: Contributions to General Algebra 7 (eds. D.~Dorninger et al.), Hoelder--Pichler--Tempsky, Wien, Teubner, Stuttgart, 1991, pp.~127--135.

\bibitem{ds2}
{\sc M.~Drmota and M.~Ska{\l}ba,} {\it Relations between polynomial roots}, Acta Arith. {\bf 71} (1995), 65--77.


\bibitem{dub}
{\sc A.~Dubickas,} {\it On the degree of  a linear form in conjugates of an
algebraic number,}
Illinois J. Math.  {\bf 46} (2002), 571--585.

\bibitem{dub2}
{\sc A.~Dubickas,} {\it 
Additive relations with conjugate algebraic numbers,} Acta Arith. {\bf 107} (2003), 35--43.

\bibitem{DHJ}
{\sc A.~Dubickas, K.~Hare and J.~Jankauskas}, {\it There are no two non-real conjugates of a Pisot number with the same imaginary part,} 
Math. Comp. {\bf 86} (2017), 935--950.

\bibitem{DJ}
{\sc A.~Dubickas and J.~Jankauskas,} {\it Simple linear relations between conjugate algebraic numbers of low degree,} 
J. Ramanujan Math. Soc. {\bf 30} (2015), 219--235.


\bibitem{smy0}
{\sc A.~Dubickas and C.~J.~Smyth,} {\it On the lines passing through two conjugates of a Salem number,}
Math. Proc. Camb. Phil. Soc. {\bf 144} (2008), 29--37.

\bibitem{fifoko}
{\sc M.~Filaseta, K.~Ford, S.~Konyagin}, {\it On an irreducibility theorem of A. Schinzel associated with coverings of the integers}, Illinois J. Math. {\bf 44} (3) (2000), 633--643.

\bibitem{fima}
{\sc M.~ Filaseta, M.~Matthews, Jr.,} {\it On the irreducibility of 0, 1-polynomials of the form $f(x)x^
n+g(x)$,} Colloq. Math., {\bf 99} (1) (2004), 1--5.

\bibitem{girst}
{\sc K.~Girstmair,} {\it
Linear relations between roots of polynomials,}
Acta Arith. {\bf 89} (1999), 53--96.


\bibitem{kurb1}
{\sc V.~A.~Kurbatov,} {\it On equations of prime degree,}
Mat. Sb., N.S. {\bf 43 (85)} (1957), 349--366 (in Russian).

\bibitem{kurb2}
{\sc V.~A.~Kurbatov,} {\it Linear dependence of conjugate elements,}
Mat. Sb., N.S. {\bf 52 (94)} (1960), 701--708 (in Russian).


\bibitem{kurb3}
{\sc V.~A.~Kurbatov,} {\it Galois extensions of prime degree and their primitive elements,}
Soviet Math. (Izv. VUZ) {\bf 21} (1977), 49--52.

\bibitem{fl1}
{\sc F.~Lalande,} {\it La relation lin\'eaire $a=b+c+\dots+t$ entre les racines d'un polyn\^ ome,} J. Th\'eor. Nombres Bordeaux {\bf 19} (2007), 473--484 (in French). 

\bibitem{fl2}
{\sc F.~Lalande,} {\it \`A propos de la relation galoisienne
$x_1=x_2+x_3$,} J. Th\'eor. Nombres Bordeaux {\bf 22} (2010), 661--673 (in French). 



\bibitem{mcsm}
{\sc  J.~McKee and C.~J.~Smyth,} {\it There are Salem numbers of every trace,} Bull. London Math. Soc. {\bf 37} (2005), 25--36.


\bibitem{nar}
{\sc W.~Narkiewicz,} {\it Elementary and analytic theory of algebraic numbers,} 3rd ed., Springer Monographs in Mathematics, Springer-Verlag, Berlin, 2004.

\bibitem{salem1}
{\sc R.~Salem,} {\it A Remarkable Class of Algebraic Numbers. Proof of a Conjecture of Vijayaraghavan,} Duke Math. J. {\bf 11} (1944), 103--107.

\bibitem{salem2}
{\sc R.~Salem,} {\it Power series with integral coefficients,} Duke Math. J. {\bf 12} (1945), 153--173.

\bibitem{schi}
{\sc A.~Schinzel}, {\it Reducibility of polynomials and covering systems of congruences,} Acta Arith. {\bf 13} (1) (1967), 91--101.


\bibitem{mcsm0}
{\sc C.~J.~Smyth,} {\it Salem numbers of negative trace,} Math. Comp. {\bf 69} (2000), 827--838.

\bibitem{smy1}
{\sc C.~J.~Smyth,} {\it Conjugate algebraic numbers on conics,}
Acta Arith. {\bf 40} (1982), 333--346.


\bibitem{smy2}
{\sc C.~J.~Smyth,} {\it Additive and multiplicative 
relations connecting conjugate
algebraic numbers,}
J. Number Theory {\bf 23} (1986), 243--254.

\bibitem{vall}
{\sc A.~Valibouze,} {\it Sur les relations entre les racines d'un polyn\^ome,} Acta Arith. {\bf 131} (2008), 1--27 (in French). 


\end{thebibliography}
\end{document}